\documentclass{amsart}
\usepackage{amsfonts}
\usepackage{amssymb}
\usepackage{times}
\usepackage{xcolor}
\usepackage{amsmath}
\usepackage{amsthm}
\usepackage{comment}
\usepackage{xcolor}

\makeatletter
\def\Ddots{\mathinner{\mkern1mu\raise\p@
\vbox{\kern7\p@\hbox{.}}\mkern2mu
\raise4\p@\hbox{.}\mkern2mu\raise7\p@\hbox{.}\mkern1mu}}
\makeatother

\newtheorem{theorem}{Theorem}[section]
\newtheorem{corollary}{Corollary}[section]
\newtheorem{lemma}{Lemma}[section]

\newtheorem{proposition}{Proposition}[section]

\newtheorem{remark}{Remark}[section]

\def\A{{\mathfrak A}\, }

\begin{document}

\title{$*$-Jordan-type maps on alternative $*$-algebras}

\thanks{The first author was supported by the Coordenação de Aperfeiçoamento de Pessoal de Nível Superior - Brasil (CAPES)-Finance 001.}

\author[Aline J. O. Andrade]{Aline J. O. Andrade}
\address{Aline Jaqueline de Oliveira Andrade, Federal University of ABC, 
dos Estados Avenue, 5001, 
09210-580, Santo Andr\'{e}, Brazil\\ 
{\em E-mail}: {\tt aline.jaqueline@ufabc.edu.br}}{}

\author[Bruno L. M. Ferreira]{Bruno L. M. Ferreira}
\address{Bruno Leonardo Macedo Ferreira, Federal University of Technology, 
Avenida Professora Laura Pacheco Bastos, 800, 
85053-510, Guarapuava, Brazil\\
{\em E-mail}: {\tt brunolmfalg@gmail.com}}{}

\author[Liudmila~Sabinina]{Liudmila~Sabinina}
\address{Liudmila~Sabinina, Autonomous University of the State of Morelos\\
Cuernavaca, Mexico.\\
{\em E-mail}: {\tt liudmila@uaem.mx}}{}

\begin{abstract}
Let $\A$ and $\A'$ be two alternative $*$-algebras with identities $1_{\A}$ and $1_{\A'}$, respectively, and $e_1$ and $e_2 = 1_{\A} - e_1$ nontrivial symmetric idempotents in $\A$. In this paper we study the characterization of multiplicative $*$-Jordan-type maps on alternative algebras. 

\vspace*{.1cm}

\vspace*{.1cm}

\end{abstract}

\maketitle

\section{Introduction and Preliminaries}
The study on characterizations of certain maps on non-associative algebraic structures has become an active and broad line of research in recent years, we can mention \cite{chang, Fer3, Fer, bruth, Fer1, Fer2, Fer4, FerGur1}.
For the case of associative structures Bre$\check{s}$ar and Fo$\check{s}$ner in \cite{brefos1, brefos2}, presented the following definition: for $a, b \in R$, where $R$ is a $*$-ring, we denote by $\{a,b\}_* = ab+ba^{*}$ and $[a, b]_{*} = ab-ba^{*}$ the $*$-Jordan product and the $*$-Lie product, respectively. In \cite{LiLuFang}, the authors proved that a map $\varphi$ between two factor von Newmann algebras is a $*$-ring isomorphism if and only if $\varphi(\{a,b\}_*) = \{\varphi(a), \varphi(b)\}_*$. In \cite{Ferco}, Ferreira and Costa extended these new products and defined two other types of applications, named multiplicative $*$-Lie $n$-map and multiplicative $*$-Jordan $n$-map and used them to impose condition such that a map between $C^*$-algebras is a $*$-ring isomorphism. Even more, in \cite{Taghavi}, Taghavi et al considered the product $a \bullet b = a^*b+ba^*$ and studied when a bijective map is additive in $*$-algebras.

With this picture in mind, in this article we will discuss when a multiplicative $*$-Jordan $n$-map is a $*$-ring isomorphism on nonassociative structure, namely, alternative algebras. As a consequence of our main result, we provide an application on alternative $W^{*}$-factor.

Throughout the paper, the ground field is assumed to be of complex numbers. Consider the product $\{ x, y\}_* = xy + yx^*$ and let us define the following sequence of polynomials, as defined in \cite{Ferco}: 
$$q_{1*}(x) = x\, \,  \text{and}\, \,  q_{n*}(x_1, x_2, \ldots , x_n) = \left\{q_{(n-1)*}(x_1, x_2, \ldots , x_{n-1}) , x_n\right\}_{*},$$
for all integers $n \geq 2$. Thus, $q_{2*}(x_1, x_2) = \left\{x_1, x_2\right\}_{*}, \ q_{3*} (x_1, x_2, x_3) = \left\{\left\{x_1, x_2\right\}_{*} , x_3\right\}_{*}$, etc. Note that $q_{2*}$ is the product introduced by Bre$\check{s}$ar and Fo$\check{s}$ner \cite{brefos1, brefos2}. Then, using the nomenclature introduced in \cite{Ferco} we have a new class of maps (not necessarily additive): $\varphi : \A \longrightarrow \A'$ is a \textit{multiplicative $*$-Jordan $n$-map} if
\[
\varphi(q_{n*} (x_1, x_2, . . . ,x_n)) =  q_{n*} (\varphi(x_1), \varphi(x_2), . . . , \varphi(x_i), . . .,\varphi(x_n)),
\]
where $n \geq 2$ is an integer.
Multiplicative $*$-Jordan $2$-map, $*$-Jordan $3$-map 
and $*$-Jordan $n$-map are collectively referred to as \textit{multiplicative $*$-Jordan-type maps}.

Let be $e_1$ a nontrivial symmetric idempotent in $\A$ and $e_2 = 1_{\A} - e_1$ where $1_{\A}$ is the identity of $\A$. Let us consider an alternative $*$-algebra $\mathfrak{A}$ and fix a nontrivial symmetric idempotent $e_{1}\in\mathfrak{A}$. It is easy to see that $(e_ia) e_j=e_i(ae_j)~(i,j=1,2)$ for all $a\in \mathfrak{A}$. Then $\mathfrak{A}$ has a Peirce decomposition
$\mathfrak{A}=\mathfrak{A}_{11}\oplus \mathfrak{A}_{12}\oplus
\mathfrak{A}_{21}\oplus \mathfrak{A}_{22},$ where
$\mathfrak{A}_{ij}=e_{i}\mathfrak{A}e_{j}$ $(i,j=1,2)$ \cite{He}, satisfying the following multiplicative relations:
\begin{enumerate}\label{asquatro}
\item [\it (i)] $\mathfrak{A}_{ij}\mathfrak{A}_{jl}\subseteq\mathfrak{A}_{il}\
(i,j,l=1,2);$
\item [\it (ii)] $\mathfrak{A}_{ij}\mathfrak{A}_{ij}\subseteq \mathfrak{A}_{ji}\
(i,j=1,2);$
\item [\it (iii)] $\mathfrak{A}_{ij}\mathfrak{A}_{kl}=0,$ if $j\neq k$ and
$(i,j)\neq (k,l),\ (i,j,k,l=1,2);$
\item [\it (iv)] $x_{ij}^{2}=0,$ for all $x_{ij}\in \mathfrak{A}_{ij}\ (i,j=1,2;~i\neq j)$.
\end{enumerate}

\begin{remark}

Observe that in the case associative $a_{ij}b_{ij} = 0$ for $i \neq j$ but in general in alternative algebras, due the property $(ii)$, we do not have this relation.
\end{remark}

The following two claims play a very important role in the further development of the paper. By definition of involution clearly we get
 
\begin{proposition}\label{obs}
$(\A_{ij})^* \subseteq \A_{ji}$ for $i,j \in \left\{1,2\right\}.$
\end{proposition}
\begin{proof}
If $a_{ij} \in \A_{ij}$ then $$a^*_{ij} = (e_i a_{ij} e_j)^* = (e_j)^*(a_{ij})^* (e_i)^* = e_j (a_{ij})^* e_i \in \A_{ji}.$$
\end{proof}
\begin{proposition}\label{Claim2} Let $\A$ and $\A'$ be two alternative $*$-algebras and $\varphi: \A \rightarrow \A'$ a bijective map which satisfies
\[
\varphi(q_{n*}(\xi, \ldots,\xi, a,b)) = q_{n*}(\varphi(\xi),\ldots,\varphi(\xi),\varphi(a),\varphi(b))
\]
for all $a, b \in \A$ and $\xi \in \left\{1_{\A}, e_1, e_2\right\}$. Let $x,y$ and $h$ be in $\A$ such that $\varphi(h) = \varphi(x) + \varphi(y)$. Then,
 given $z \in \A$,
$$
\begin{aligned}
\varphi(q_{n*}(t,...,t,h,z)) = \varphi(q_{n*}(t,...,t,x,z)) + \varphi(q_{n*}(t,...,t,y,z))
\end{aligned}
$$
and
$$
\begin{aligned}
\varphi(q_{n*}(t,...,t,z,h)) = \varphi(q_{n*}(t,...,t,z,x)) + \varphi(q_{n*}(t,...,t,z,y))
\end{aligned}
$$
for $t = 1_\A$ or $t = e_i$, $i = 1,2$.

\end{proposition}

\begin{proof}
Using the definition of $\varphi$ and multilinearity of $q_{n*}$ we obtain
$$
\begin{aligned}
\varphi(q_{n*}(t,...,t,h,z)) &= q_{n*}(\varphi(t),...,\varphi(t),\varphi(h),\varphi(z)) \\
&= q_{n*}(\varphi(t),...,\varphi(t),\varphi(x)+\varphi(y),\varphi(z)) \\
&= q_{n*}(\varphi(t),...,\varphi(t),\varphi(x),\varphi(z)) + q_{n*}(\varphi(t),...,\varphi(t),\varphi(y),\varphi(z)) \\
&= \varphi(q_{n*}(t,...,t,x,z)) + \varphi(q_{n*}(t,...,t,y,z)).
\end{aligned}
$$
In a similar way we have
$$
\begin{aligned}
\varphi(q_{n*}(t,...,t,z,h)) &= \varphi(q_{n*}(t,...,t,z,x)) + \varphi(q_{n*}(t,...,t,z,y)).
\end{aligned}
$$

\end{proof}

\section{Main theorem}

We shall prove as follows a part of the the main result of this paper:

\begin{theorem}\label{mainthm1} 
Let $\A$ and $\A'$ be two alternative $*$-algebras with identities $I_{\A}$ and $1_{\A'}$, respectively, and $e_1$ and $e_2 = 1_{\A} - e_1$ nontrivial symmetric idempotents in $\A$. Suppose that $\A$ satisfies
\begin{eqnarray*}
  &&\left(\spadesuit\right) \ \ \  \ \ \  x (\A e_i) = \left\{0\right\} \ \ \  \mbox{implies} \ \ \ x = 0.
\end{eqnarray*}
Even more, suppose that $\varphi: \A \rightarrow \A'$ is a bijective unital map which satisfies
\[
\varphi(q_{n*}(\xi, \ldots,\xi, a,b)) = q_{n*}(\varphi(\xi),\ldots,\varphi(\xi),\varphi(a),\varphi(b))
\]
for all $a, b \in \A$ and $\xi \in \left\{1_{\A}, e_1, e_2\right\}$. Then $\varphi$ is $*$-additive.
\end{theorem}

It is easy to see that any prime alternative $*$-algebra over ground field of characteristic different $2,3$  satisfies $\left(\spadesuit\right)$ hence we have  

\begin{corollary}
Let $\A$ be prime alternative $*$-algebra and $\A'$ an alternative $*$-algebra with identities $1_{\A}$ and $1_{\A'}$, respectively, and $e_1$ and $e_2 = 1_{\A} - e_1$ nontrivial symmetric idempotents in $\A$. Suppose that $\varphi: \A \rightarrow \A'$ is a bijective unital map which satisfies
\[
\varphi(q_{n*}(\xi, \ldots,\xi, a,b)) = q_{n*}(\varphi(\xi),\ldots,\varphi(\xi),\varphi(a),\varphi(b))
\]
for all $a, b \in \A$ and $\xi \in \left\{1_{\A}, e_1, e_2\right\}$. Then $\varphi$ is $*$-additive.

\end{corollary}

The following lemmas have the same hypotheses as the Theorem \ref{mainthm1} and we need them to prove the $*$-additivity of $\varphi$. 

\begin{lemma}\label{Claim1}  $\varphi(0) = 0$.
\end{lemma}

\begin{proof}
Since $\varphi$ is surjective, there exists $x \in \A$ such that $\varphi(x) = 0$. Then
$$
0 = q_{n*}(\varphi(e_1),...,\varphi(e_1),\varphi(e_2),\varphi(x)) = \varphi(q_{n*}(e_1,...,e_1,e_2,x)) = \varphi(0).
$$
Therefore, $\varphi(0) = 0$.
\end{proof}

\begin{lemma}\label{lema1} For any $a_{11} \in \A_{11}$ and $b_{22} \in \A_{22}$, we have 
$$\varphi(a_{11} + b_{22}) = \varphi(a_{11}) + \varphi(b_{22}).$$
\end{lemma}
\begin{proof}
Since $\varphi$ is surjective, given $\varphi(a_{11})+\varphi(b_{22}) \in \A'$ there exists $t \in \A$ such that 
$\varphi(t) = \varphi(a_{11})+\varphi(b_{22})$, with $t=t_{11}+t_{12}+t_{21}+t_{22}$. Now, by Proposition \ref{Claim2}
$$
\varphi(q_{n*}(e_i,...,e_i,t)) = \varphi(q_{n*}(e_i,...,e_i,a_{11})) + \varphi(q_{n*}(e_i,...,e_i,b_{22})),
$$
with $i=1,2$. It follows that
$$
\varphi(2^{n-2}(e_it + te_i)) = \varphi(2^{n-2}(e_ia_{11} + a_{11}e_i)) + \varphi(2^{n-2}(e_ib_{22} + b_{22}e_i)).
$$
Using the injectivity of $\varphi$ we obtain
$2^{n-2}(2t_{11} + t_{12} + t_{21}) = 2^{n-2}(2a_{11})$ and
$2^{n-2}(2t_{22} + t_{12} + t_{21}) = 2^{n-2}(2b_{22}).$ Then $t_{11}=a_{11}$, $t_{22}=b_{22}$ and $t_{12}=t_{21}=0$.
\end{proof}

\begin{lemma}\label{lema2}
For any $a_{12} \in \A_{12}$ and $b_{21} \in \A_{21}$, we have $\varphi(a_{12} + b_{21}) = \varphi(a_{12}) + \varphi(b_{21})$.
\end{lemma}

\begin{proof}
Since $\varphi$ is surjective, given $\varphi(a_{12})+\varphi(b_{21}) \in \A'$ there exists $t \in \A$ such that 
$\varphi(t) = \varphi(a_{12})+\varphi(b_{21})$, with $t=t_{11}+t_{12}+t_{21}+t_{22}$. Now, by Proposition \ref{Claim2}
$$
\begin{aligned}
\varphi\left(q_{n*}\left(e_1,...,e_1,\frac{1}{2^{n-1}}e_1,t\right)\right) &= \varphi\left(q_{n*}\left(e_1,...,e_1,\frac{1}{2^{n-1}}e_1,a_{12}\right)\right) \\
& +\varphi\left(q_{n*}\left(e_1,...,e_1,\frac{1}{2^{n-1}}e_1,b_{21}\right)\right) \\ 
&= \varphi(e_1a_{12} + a_{12}e_1) + \varphi(e_1b_{21} + b_{21}e_1) \\
&= \varphi(a_{12}) + \varphi(b_{21}) = \varphi(t).
\end{aligned}
$$
Since $\varphi$ is injective, $e_1t + te_1 = t,$ that is,
$2t_{11} + t_{12} + t_{21} = t_{11}+t_{12}+t_{21}+t_{22}.
$ Then $t_{11} = t_{22} = 0$.

Now, observe that, for $c_{12}\in \A_{12},$ $q_{n*}\left(e_2, \ldots, e_2,\dfrac{1}{2^{n-2}}e_2,t,c_{12}\right) = t_{21}c_{12}+t_{12}c_{12}+c_{12}t_{21}^*+c_{12}t_{12}^*.$ Thus, $q_{n*}\left(e_2,\ldots,e_2,\dfrac{1}{2^{n-2}}e_2,a_{12},c_{12}\right) = a_{12}c_{12}+c_{12}a_{12}^* \in \A_{11}+\A_{21}$ and $q_{n*}\left(e_2, \ldots,e_2, \dfrac{1}{2^{n-2}}e_2,b_{21},c_{12}\right)=b_{21}c_{12}+c_{12}b_{21}^* \in \A_{21}+\A_{22}.$ Moreover, 
\[q_{n*}\left(e_2, \ldots, e_2, \dfrac{1}{2^{n-2}}e_2, q_{n*}\left(e_2, \ldots, e_2,\dfrac{1}{2^{n-2}}e_2,t,c_{12}\right),e_2\right) =2(t_{21}c_{12}+(t_{21}c_{12})^*),
\]
thus 
\begin{align*}
\varphi& \left(q_{n*}\left(e_2,\ldots,e_2, \dfrac{1}{2^{n-2}}e_2,q_{n*}\left(e_2, \ldots,e_2, \dfrac{1}{2^{n-2}}e_2,t,c_{12}\right),e_2\right)\right) \\
=& \ \varphi\left(q_{n*}\left(e_2,\ldots,e_2,\dfrac{1}{2^{n-2}}e_2,q_{n*}\left(e_2, \ldots,e_2,\dfrac{1}{2^{n-2}}e_2,a_{12},c_{12}\right),e_2\right)\right)\\
& \ +\varphi\left(q_{n*}\left(e_2,\ldots,e_2,\dfrac{1}{2^{n-2}}e_2,q_{n*}\left(e_2, \ldots,e_2,\dfrac{1}{2^{n-2}}e_2,b_{21},c_{12}\right),e_2\right)\right) \\
=& \ \varphi(0) + \varphi(2(b_{21}c_{12}+(b_{21}c_{12})^*))=\varphi(2(b_{21}c_{12}+(b_{21}c_{12})^*)).
\end{align*}
Therefore, $(t_{21}-b_{21})c_{12}+c_{12}^*(t_{21}^*-b_{21}^*)=0,$ for all $c_{12}\in \A_{12}.$ Consider $ic_{12}\in \A_{12},$ we have $(t_{21}-b_{21})c_{12}-c_{12}^*(t_{21}^*-b_{21}^*)=0.$ Thus, $(t_{21}-b_{21})c_{12}=0$ implies $(t_{21}-b_{21})(\A e_2) = 0.$ Using ($\spadesuit$), we obtain $t_{21}-b_{21}=0,$ that is, $t_{21}=b_{21}.$ Note that, with similar calculations, if we replace $c_{12}$ by $c_{21}$ and $e_1$ by $e_2,$ obtain that $t_{12} = a_{12}.$
\end{proof}

\begin{lemma}\label{lema3}
For any $a_{11} \in \A_{11}$, $b_{12} \in \A_{12}$, $c_{21} \in \A_{21}$ and $d_{22} \in \A_{22}$ we have 
$$\varphi(a_{11} + b_{12} + c_{21}) = \varphi(a_{11}) + \varphi(b_{12})+ \varphi(c_{21})$$
and
$$\varphi(b_{12} + c_{21} + d_{22}) = \varphi(b_{12})+ \varphi(c_{21}) + \varphi(d_{22}).$$
\end{lemma}

\begin{proof}
Since $\varphi$ is surjective, given $\varphi(a_{11})+\varphi(b_{12})+\varphi(c_{21}) \in \A'$ there exists $T \in \A$ such that 
$\varphi(t) = \varphi(a_{11})+\varphi(b_{12})+\varphi(c_{21})$, with $t=t_{11}+t_{12}+t_{21}+t_{22}$. Now, observing that
$q_{n*}(e_2,...,e_2,a_{11})=0$ and using Proposition \ref{Claim2} and Lemma \ref{lema2}, we obtain
$$
\begin{aligned}
\varphi(q_{n*}(e_2,...,e_2,t)) 
&= \varphi(q_{n*}(e_2,...,e_2,a_{11})) + \varphi(q_{n*}(e_2,...,e_2,b_{12})) \\
&+ \ \varphi(q_{n*}(e_2,...,e_2,c_{21})) \\ 
&= \ \varphi(q_{n*}(e_2,...,e_2,b_{12}) + q_{n*}(e_2,...,e_2,c_{21})).
\end{aligned}
$$
By injectivity of $\varphi$ we have $q_{n*}(e_2,...,e_2,t)=q_{n*}(e_2,...,e_2,b_{12}) + q_{n*}(e_2,...,e_2,c_{21}),$ that is,
$2t_{22} + t_{12} + t_{21} = b_{12} + c_{21}.$ Therefore, $t_{22}=0$, $t_{12}=b_{12}$ and $t_{21}=c_{21}$.
Again, observing that 
\[
q_{n*}(1_\A,...,1_\A,e_1-e_2,b_{12})=q_{n*}(1_\A,...,1_\A,e_1-e_2,c_{21})=0
\]
and using 
Proposition \ref{Claim2}, we obtain
$$
\begin{aligned}
\varphi(q_{n*}(1_\A,...,1_\A,e_1-e_2,t)) &= \varphi(q_{n*}(1_\A,...,1_\A,e_1-e_2,a_{11}))\\&+\varphi(q_{n*}(1_\A,...,1_\A,e_1-e_2,b_{12})) \\
&+ \varphi(q_{n*}(1_\A,...,1_\A,e_1-e_2,c_{21})) \\
&= \varphi(q_{n*}(1_\A,...,1_\A,e_1-e_2,a_{11})).
\end{aligned}
$$
By injectivity of $\varphi$ we have $q_{n*}(1_\A,...,1_\A,e_1-e_2,t) = q_{n*}(1_\A,...,1_\A,e_1-e_2,a_{11}),
$ that is,
$2t_{11} - 2t_{22} = 2a_{11}.$ Therefore, $t_{11}=a_{11}$.

The other identity we obtain in a similar way.
\end{proof}

\begin{lemma}\label{lema4}
For any $a_{11} \in \A_{11}$, $b_{12} \in \A_{12}$, $c_{21} \in \A_{21}$ and $d_{22} \in \A_{22}$ we have 
$$\varphi(a_{11} + b_{12} + c_{21} + d_{22}) = \varphi(a_{11}) + \varphi(b_{12}) + \varphi(c_{21}) + \varphi(d_{22}).$$
\end{lemma}
\begin{proof}
Since $\varphi$ is surjective, given $\varphi(a_{11})+\varphi(b_{12})+\varphi(c_{21})+\varphi(d_{22}) \in \A'$ there exists $T \in \A$ such that 
$\varphi(t) = \varphi(a_{11})+\varphi(b_{12})+\varphi(c_{21})+\varphi(d_{22})$, with $t=t_{11}+t_{12}+t_{21}+t_{22}$. Now, observing that
$q_{n*}(e_1,...,e_1,e_{22})=0$ and using Proposition \ref{Claim2} and Lemma \ref{lema3}, we obtain
$$
\begin{aligned}
\varphi(q_{n*}(e_1,...,e_1,t)) &= \varphi(q_{n*}(e_1,...,e_1,a_{11})) + \varphi(q_{n*}(e_1,...,e_1,b_{12})) \\
&+ \varphi(q_{n*}(e_1,...,e_1,c_{21})) + \varphi(q_{n*}(e_1,...,e_1,d_{22})) \\ 
&= \varphi(q_{n*}(e_1,...,e_1,a_{11})) + \varphi(q_{n*}(e_1,...,e_1,b_{12})) \\
&+ \varphi(q_{n*}(e_1,...,e_1,c_{21})) \\
&= \varphi(q_{n*}(e_1,...,e_1,a_{11}) + q_{n*}(e_1,...,e_1,b_{12}) + q_{n*}(e_1,...,e_1,c_{21})).
\end{aligned}
$$
By injectivity of $\varphi$ we have $q_{n*}(e_1,...,e_1,t) = q_{n*}(e_1,...,e_1,a_{11}) + q_{n*}(e_1,...,e_1,b_{12}) + q_{n*}(e_1,...,e_1,c_{21}),$ that is, $2t_{11} + t_{12} + t_{21} = 2a_{11} + b_{12} + c_{21}.$ Therefore, $t_{11}=a_{11}$, $t_{12}=b_{12}$ and $t_{21}=c_{21}$.

In a similar way, using $q_{n*}(e_2,...,e_2,t)$, we obtain
$2t_{22} + t_{12} + t_{21} = 2d_{22} + b_{12} + c_{21}$ and then $t_{22}=d_{22}$.
\end{proof}


\begin{lemma}\label{lema5.1}
For every $a_{12}, b_{12} \in \A_{12}$ and $c_{21}, d_{21} \in \A_{21}$ we have  
$$\varphi(a_{12}b_{12}+a^{*}_{
12}) = \varphi(a_{12}b_{12})+\varphi(a^{*}_{12})$$
and
$$\varphi(c_{21}d_{21}+c^{*}_{
21}) = \varphi(c_{21}d_{21})+\varphi(c^{*}_{
21}).$$
\end{lemma}
\begin{proof}
Since 
$q_{n*}\left(1_{\A},...,1_{\A},a_{12}, \dfrac{1}{2^{n-2}}(e_2+b_{12})\right) = a_{12}+a_{12}b_{12} + a_{12}^*+b_{12}a_{12}^*$
we get from Lemma \ref{lema4} that
\begin{align*}
\varphi(a_{12})+ & \varphi(a_{12}b_{12} + a^{*}_{12}) + \varphi(b_{12}a^{*}_{12})\\ 
=& \ \varphi(a_{12} + a_{12}b_{12} + a^{*}_{12} + b_{12}a^{*}_{12}) \\
=& \ \varphi\left( q_{n*}\left(1_{\A},...,1_{\A},a_{12}, \dfrac{1}{2^{n-2}}(e_2+b_{12})\right) \right) \\
=& \ \left( q_{n*}\left(\varphi(1_{\A}),...,\varphi(1_{\A}), \varphi(a_{12}), \varphi(e_2) + \varphi\left(\dfrac{1}{2^{n-2}}b_{12}\right) \right)\right) \\
=& \ \varphi(a_{12}) + \varphi(a_{12}^*) + \varphi(a_{12}b_{12}) + \varphi(b_{12}a_{12}^*),
\end{align*}
which implies $\varphi(a_{12}b_{12} + a^{*}_{12}) = \varphi(a_{12}b_{12}) + \varphi(a^{*}_{12}).$
Similarly, we prove the other case using the identity
$q_{n*}\left(1_{\A},...,1_{\A},c_{21}, \dfrac{1}{2^{n-2}}(e_1+d_{21})\right) = c_{21}+c_{21}d_{21} + c_{21}^*+d_{21}c_{21}^*$
\end{proof}

\begin{lemma}\label{lema5.2}
For all $a_{ij}, b_{ij} \in \A_{ij},$ $1 \leq i \neq j \leq 2$, we have
$$\varphi(a_{ij} + b_{ij}) = \varphi(a_{ij}) + \varphi(b_{ij}).$$
\end{lemma}
\begin{proof}
Since
$q_{n*}\left(1_{\A},...,1_{\A},\dfrac{1}{2^{n-2}}(e_i+a_{ij}), e_j+b_{ij}\right) = a_{ij} + b_{ij} + a_{ij}^* + a_{ij}b_{ji}+b_{ij}a_{ij}^*$
we get from Lemma \ref{lema5.1} that
\begin{align*}
\varphi & (a_{ij}+b_{ij}) \varphi(a_{ij}^* + a_{ij}b_{ij}) + \varphi(b_{ij}a_{ij}^*) \\
=& \ \varphi \left( q_{n*}(1_{\A},...,1_{\A},\dfrac{1}{2^{n-2}}(e_i+a_{ij}),e_j+b_{ij}) \right)\\
=& \  q_{n*}\left(\varphi(1_{\A}),...,\varphi(1_{\A}),\varphi\left(\dfrac{1}{2^{n-2}}(e_i+a_{ij})\right), \varphi(e_j+b_{ij}) \right)\\
=& \  q_{n*}\left(\varphi(1_{\A}),...,\varphi(1_{\A}),\varphi(e_i)\varphi\left(\dfrac{1}{2^{n-2}}a_{ij}\right), \varphi(e_j)+\varphi(b_{ij}) \right)\\
=& \ \varphi(b_{ij}) + \varphi(a_{ij} + a_{ij}^*) +  \varphi(a_{ij}b_{ij}+b_{ij}a_{ij}^*)\\ 
=& \ \varphi(b_{ij}) + \varphi(a_{ij}) + \varphi(a_{ij}^*) + \varphi(a_{ij}b_{ij}) + \varphi(b_{ij}a_{ij}^*), 
\end{align*}
which implies $\varphi(a_{ij} + b_{ij}) = \varphi(b_{ij}) + \varphi(a_{ij})$.
\end{proof}

\begin{lemma}\label{lema6}
For all $a_{ii}, b_{ii} \in \A_{ii}$, we have $\varphi(a_{ii} + b_{ii}) = \varphi(a_{ii}) + \varphi(b_{ii})$ for $i \in \left\{1,2\right\}.$
\end{lemma}
\begin{proof}
Since $\varphi$ is surjective, given $\varphi(a_{ii})+\varphi(b_{ii}) \in \A'$, $i=1,2$, there exists $t \in \A$ such that 
$\varphi(t) = \varphi(a_{ii})+\varphi(b_{ii})$, with $t=t_{11}+t_{12}+t_{21}+t_{22}$. By Proposition \ref{Claim2}, for $j \neq i$,
$$
\varphi(q_{n*}(e_j,...,e_j,t)) = \varphi(q_{n*}(e_j,...,e_j,a_{ii})) + \varphi(q_{n*}(e_j,...,e_j,b_{ii})) = 0.
$$
Then, $t_{ij}=t_{ji}=t_{jj}=0$. We just have to show that $t_{ii} = a_{ii} + b_{ii}$. Given $c_{ij} \in \A_{ij}$, using Lemma \ref{lema5.2} and  
Proposition \ref{Claim2} we have
$$
\begin{aligned}
\varphi(q_{n*}(e_i,...,e_i,t,c_{ij})) &= \varphi(q_{n*}(e_i,...,e_i,a_{ii},c_{ij}))+\varphi(q_{n*}(e_i,...,e_i,b_{ii},c_{ij})) \\
                                      &= \varphi(q_{n*}(e_i,...,e_i,a_{ii},c_{ij}) + q_{n*}(e_i,...,e_i,b_{ii},c_{ij})).
\end{aligned}
$$
By injectivity of $\varphi$ we obtain
$$
q_{n*}(e_i,...,e_i,t,c_{ij}) = q_{n*}(e_i,...,e_i,a_{ii},c_{ij}) + q_{n*}(e_i,...,e_i,b_{ii},c_{ij}),
$$
that is,
$$
(t_{ii} - a_{ii} - b_{ii})c_{ij} = 0.
$$
Finally, by $\left(\spadesuit\right)$ we conclude that $t_{ii} = a_{ii} + b_{ii}$.
\end{proof}

Now we are able to show that $\varphi$ preserves $*$-addition.

Using Lemmas \ref{lema4}, \ref{lema5.2}, \ref{lema6} it is easy to see that $\varphi$ is additive.
Besides, on the one hand, since $\varphi$ is additive it follows that 
$$\varphi(a + a^{*}) = \varphi(a) + \varphi(a^{*}).$$
On the other hand, by additivity of $\varphi$,
$$
\begin{aligned}
2^{n-2}\varphi(a + a^{*}) &= \varphi(2^{n-2}(a + a^{*}))  = \varphi(q_{n*}(1_{\A},...,1_{\A},a,1_{\A})) \\
               &= q_{n*}(1_{\A'},...,1_{\A'},\varphi(a),1_{\A'}) = 2^{n-2}(\varphi(a) + \varphi(a)^{*}).
\end{aligned}
$$
Therefore $\varphi(a^{*}) = \varphi(a)^{*}$ and Theorem \ref{mainthm1} is proved.

\vspace{0,5cm}
Now we focus our attention on investigate the problem of when $\varphi$ is a $*$-ring isomorphism. We prove the following result:

\begin{theorem}\label{mainthm2}
Let $\A$ and $\A'$ be two alternative $*$-algebras with identities $1_{\A}$ and $1_{\A'}$, respectively, and $e_1$ and $e_2 = 1_{\A} - e_1$ nontrivial symmetric idempotents in $\A$. Suppose that $\A$ and $\A'$ satisfy:
\begin{eqnarray*}
  &&\left(\spadesuit\right) \ \ \  \ \ \  x (\A e_i) = \left\{0\right\} \ \ \  \mbox{implies} \ \ \ x = 0 
  \\
  \mbox{ and}
	\\&& \left(\clubsuit \right) \ \ \  \ y (\A'\varphi(e_i)) = \left\{0\right\} \ \ \  \mbox{implies} \ \ \ y = 0.
\end{eqnarray*}
If $\varphi: \A \rightarrow \A'$ is a bijective unital map which satisfies
\[
\varphi(q_{n*}(\xi, \ldots, \xi, a ,b)) = q_{n*}(\varphi(\xi), \ldots, \varphi(\xi),\varphi(a),\varphi(b)),
\]
 for all $a, b \in \A$ and $e \in \left\{1_{\A},e_1, e_2\right\}$ then $\varphi$ is $*$-ring isomorphism.
 
\end{theorem}

Again it is easy to see that prime alternative algebras over ground field of characteristic different $2,3$ satisfy $\left(\spadesuit\right)$ and $\left(\clubsuit \right)$ hence we get

\begin{corollary}
Let $\A$ and $\A'$ be two prime alternative $*$-algebras with identities $1_{\A}$ and $1_{\A'}$, respectively, and $e_1$ and $e_2 = 1_{\A} - e_1$ nontrivial symmetric idempotents in $\A$. If $\varphi: \A \rightarrow \A'$ is a bijective unital map which satisfies
\[
\varphi(q_{n*}(\xi, \ldots,\xi, a,b)) = q_{n*}(\varphi(\xi),\ldots,\varphi(\xi),\varphi(a),\varphi(b))
\]
for all $a, b \in \A$ and $\xi \in \left\{1_{\A}, e_1, e_2\right\},$ then $\varphi$ is $*$-ring isomorphism.

\end{corollary}

Now since $\varphi$ is $*$-additive, by Theorem \ref{mainthm1},  it is enough to verify that $\varphi(ab) = \varphi(a)\varphi(b)$. Firstly, let us prove the following lemmas:

\begin{lemma}\label{lemam0}
 $f_i = \varphi(e_i)$ is an idempotent in $\A'$, with $i \in \{1,2\}$.
\end{lemma}
\begin{proof}
By additivity of $\varphi$ we have
 $$
 \begin{aligned}
 2^{n-1}f_i &= 2^{n-1}\varphi(e_i) = \varphi(2^{n-1}e_i) = \varphi(q_{n*}(1_\A, ...,1_\A,e_i,e_i)) \\
            &= q_{n*}(1_{\A'}, ...,1_{\A'},\varphi(e_i),\varphi(e_i)) = 2^{n-1}\varphi(e_i)\varphi(e_i) = 2^{n-1}f_if_i.
 \end{aligned}
 $$
 Therefore, $f_if_i = f_i$.
\end{proof}

\begin{lemma}\label{lemam1}
If $x \in \A_{ij}$ then $\varphi(x) \in \A'_{ij}$.
\end{lemma}
\begin{proof}
Firstly, given $x \in \A_{ij}$, with $i \neq j$, we observe that
$$
\begin{aligned}
2^{n-2}\varphi(x) &= \varphi(2^{n-2}x) = \varphi(q_{n*}(e_j,...,e_j,x)) = q_{n*}(\varphi(e_j),...,\varphi(e_j),\varphi(x)) \\
                  &= 2^{n-2}(f_j\varphi(x) + \varphi(x)f_j),
\end{aligned}
$$
that is,  since $\varphi$ is unital we have $f_i\varphi(x)f_i = f_j\varphi(x)f_j = 0$.
Even more,
$$
\begin{aligned}
0 &= \varphi(q_{n*}(e_i,...,e_i,x,e_i)) = q_{n*}(f_i,...,f_i,\varphi(x),f_i) \\
  &= 2^{n-3}(f_i\varphi(x)f_i + \varphi(x)f_i + f_i\varphi(x)^*f_i + f_i\varphi(x)^*).
\end{aligned}
$$
Multiplying left side by $f_j$ we obtain $f_j\varphi(x)f_i = 0$. Therefore, $\varphi(x) \in \A'_{ij}$. 
In a similar way, if $x \in \A_{ii}$ we conclude that $\varphi(x) \in \A'_{ii}$.
\end{proof}

\begin{lemma}\label{lemam2}
If $a_{ii} \in \A_{ii}$ and $b_{ij} \in \A_{ij}$, with $i \neq j$, then $\varphi(a_{ii}b_{ij})=\varphi(a_{ii})\varphi(b_{ij})$.
\end{lemma}
\begin{proof}
Let $a_{ii} \in \A_{ii}$ and $b_{ij} \in \A_{ij}$, with $i \neq j$. Then, by Lemma \ref{lemam1} and additivity of $\varphi$,
$$
\begin{aligned}
2^{n-2}\varphi(a_{ii}b_{ij}) &= \varphi(2^{n-2}a_{ii}b_{ij}) = \varphi(q_{n*}(e_i,...,e_i,a_{ii},b_{ij})) \\
            &= q_{n*}(\varphi(e_i),...,\varphi(e_i),\varphi(a_{ii}),\varphi(b_{ij})) = 2^{n-2}\varphi(a_{ii})\varphi(b_{ij}).
\end{aligned}
$$
Therefore, $
\varphi(a_{ii}b_{ij}) = \varphi(a_{ii})\varphi(b_{ij}).$
\end{proof}

\begin{lemma}\label{lemam3}
If $a_{ii},b_{ii} \in \A_{ii}$ then $\varphi(a_{ii}b_{ii}) = \varphi(a_{ii})\varphi(b_{ii})$.
\end{lemma}
\begin{proof}
Let $x$ be an element of $\A_{ij}$, with $i \neq j$. Using Lemmas~\ref{lemam1},~\ref{lemam2} and the flexibility of alternative algebras we obtain
\begin{eqnarray*}
\varphi(a_{ii}b_{ii})\varphi(x) =& \  \varphi((a_{ii}b_{ii})x) = \varphi(a_{ii}(b_{ii}x)) = \varphi(a_{ii})\varphi(b_{ii}x) \\ =&  \varphi(a_{ii})(\varphi(b_{ii})\varphi(x)) =(\varphi(a_{ii})\varphi(b_{ii}))\varphi(x)
\end{eqnarray*}
that is,
$$
(\varphi(a_{ii}b_{ii}) - \varphi(a_{ii})\varphi(b_{ii}))\varphi(x) = 0.
$$
Now, by Lemma \ref{lemam1}, since $\varphi(x) \in \A'_{ij}$ and $\varphi(a_{ii}b_{ii}) - \varphi(a_{ii})\varphi(b_{ii}) \in \A'_{ii}$, we have
$$
(\varphi(a_{ii}b_{ii}) - \varphi(a_{ii})\varphi(b_{ii}))(\A'\varphi(e_j)) = \{0\}.
$$
Finally, $\left(\clubsuit \right)$ ensures that $\varphi(a_{ii}b_{ii}) = \varphi(a_{ii})\varphi(b_{ii})$.
\end{proof}

\begin{lemma} \label{lemam4}
If $a_{ij} \in \A_{ij}$ and $b_{ji} \in \A_{ji}$, with $i \neq j$, then $\varphi(a_{ij}b_{ji}) = \varphi(a_{ij})\varphi(b_{ji})$.
\end{lemma}
\begin{proof}
Let $a_{ij} \in \A_{ij}$ and $b_{ji} \in \A_{ji}$, with $i \neq j$. Then, by Lemma \ref{lemam1} and additivity of $\varphi$, 
$$
\begin{aligned}
2^{n-3}\varphi(a_{ij}b_{ji}) &= \varphi(2^{n-3}a_{ij}b_{ji}) = \varphi(q_{n*}(e_i,...,e_i,a_{ij},b_{ji})) \\
     &= q_{n*}(\varphi(e_i),...,\varphi(e_i),\varphi(a_{ij}),\varphi(b_{ji})) = 2^{n-3}\varphi(a_{ij})\varphi(b_{ji}).
\end{aligned}
$$
Therefore, $\varphi(a_{ij}b_{ji}) = \varphi(a_{ij})\varphi(b_{ji}).$

\end{proof}

\begin{lemma}\label{lemam5}
If $a_{ij} \in \A_{ij}$ and $b_{jj} \in \A_{jj}$, with $i \neq j$, then $\varphi(a_{ij}b_{jj}) = \varphi(a_{ij})\varphi(b_{jj})$
\end{lemma}
\begin{proof}
Let $x$ be an element of $\A_{ji}$, with $i \neq j$. Using Lemmas \ref{lemam2} and \ref{lemam4} we obtain
$$
\varphi(a_{ij}b_{jj})\varphi(x) = \varphi(a_{ij}b_{jj}x) = \varphi(a_{ij})\varphi(b_{jj}x) = \varphi(a_{ij})\varphi(b_{jj})\varphi(x),
$$
that is,
$$
(\varphi(a_{ij}b_{jj}) - \varphi(a_{ij})\varphi(b_{jj}))\varphi(x) = 0.
$$
Now, by Lemma \ref{lemam1}, since $\varphi(x) \in \A'_{ji}$ and $\varphi(a_{ij}b_{jj}) - \varphi(a_{ij})\varphi(b_{jj}) \in \A'_{ij}$, we have
$$
(\varphi(a_{ij}b_{jj}) - \varphi(a_{ij})\varphi(b_{jj})(\A'\varphi(e_i)) = \{0\}.
$$
Finally, $\left(\clubsuit \right)$ ensures that $\varphi(a_{ij}b_{jj}) = \varphi(a_{ij})\varphi(b_{jj})$.
\end{proof}

As in general $a_{ij}b_{ij} \neq 0$ for $i \neq j$ in alternative $*$-algebras, we need to prove.

\begin{lemma}\label{lemam6}
If $a_{ij} \in \A_{ij}$ and $b_{ij} \in \A_{ij}$, with $i \neq j$, then $\varphi(a_{ij}b_{ij}) = \varphi(a_{ij})\varphi(b_{ij})$
\end{lemma}
\begin{proof}
Let $a_{ij}, b_{ij} \in \A_{ij}$, with $i \neq j$. Then, by Lemma \ref{lemam1} and additivity of $\varphi$,
$$
\begin{aligned}
2^{n-2}\varphi(a_{ij}b_{ij}) &= \varphi(2^{n-2}a_{ij}b_{ij}) = \varphi(q_{n*}(e_i,...,e_i,a_{ij},b_{ij})) \\
            &= q_{n*}(\varphi(e_i),...,\varphi(e_i),\varphi(a_{ij}),\varphi(b_{ij})) = 2^{n-2}\varphi(a_{ij})\varphi(b_{ij}).
\end{aligned}
$$
Therefore, $\varphi(a_{ij}b_{ij}) = \varphi(a_{ij})\varphi(b_{ij}).$
\end{proof}

Thus, by additivity of $\varphi$, proved in the Theorem $\ref{mainthm1}$, and the lemmas above we conclude that $\varphi(ab) = \varphi(a)\varphi(b)$. Therefore $\varphi$ is a $*$-ring isomorphism.

\section{Application}

A complete normed alternative complex $*$-algebra $\A$ is called of alternative
$C^{*}$-algebra if it satisfies the condition: $\left\|a^{*}a\right\| = \left\|a\right\|^2$, for all elements $a \in \A$. Alternative $C^{*}$-algebras are non-associative generalizations of
$C^{*}$-algebras and appear in various areas in Mathematics (see more details in the references \cite{Miguel1} and \cite{Miguel2}). 
An alternative $C^{*}$-algebra $A$ is called of alternative $W^{*}$-algebra if it is a dual Banach space and a prime alternative $W^{*}$-algebra is called
alternative $W^{*}$-factor. It is well known that non-zero alternative $W^{*}$-algebras are unital. 

\begin{theorem}
Let $\A$ and $\A'$ be two alternative $W^{*}$-factors. If a map $\varphi : \A \rightarrow \A'$ satisfies
$$\varphi(q_{n*}(\xi, \ldots, \xi, a,b)) =
q_{n*}(\varphi(\xi), \ldots, \varphi(\xi), \varphi(a),\varphi(b)),$$
for all $a,b\in \A$ and $\xi \in \{1_\A, e_1, e_2\},$ then $\varphi$ is an additive $*$-ring isomorphism.  
\end{theorem}

\begin{corollary}
Let $\A$ and $\A'$ be two alternative $W^{*}$-factors. In this case, $\varphi$ is a nonlinear $*$-Jordan-type map if and only if $\varphi$ is an additive $*$-ring isomorphism.
  
\end{corollary}

\section{Conflict of interest}
On behalf of all authors, the corresponding author states that there is no conflict of interest.

\end{document}